\documentclass[preprint]{elsarticle}

\usepackage{amssymb}
\usepackage{amsthm}

\newcommand{\mb}{\mathbb} 
\newcommand{\CSP}{\textsc{Continuous Skolem Problem}\mbox{}}
\newtheorem{thm}{Theorem}
\newtheorem{lem}[thm]{Lemma}

\newtheorem{open}[thm]{Open Problem}

\newcommand{\C}{\mathbb C}
\newcommand{\R}{\mathbb R}
\newcommand{\PR}{\mathbb{R}_{\geq 0}} 
\newcommand{\N}{\mathbb N}
\newcommand{\Z}{\mathbb Z}
\newcommand{\Q}{\mathbb Q}

\begin{document}

\begin{frontmatter}

\title{The Continuous Skolem-Pisot Problem: On the Complexity of Reachability
for Linear Ordinary Differential Equations}

\author[]{Paul C. Bell\corref{cor1}}
\ead{p.c.bell@liverpool.ac.uk}

\author[]{Jean-Charles Delvenne}
\ead{jean-charles.delvenne@uclouvain.be}

\author[]{Rapha\"el M. Jungers}
\ead{raphael.jungers@uclouvain.be}

\author[]{Vincent D. Blondel}
\ead{vincent.blondel@uclouvain.be}

\cortext[cor1]{Corresponding Author}

\address{Department of Mathematical Engineering, Centre CESAME,\\
Universit\'e catholique de Louvain,
B\^atiment Euler,
Avenue Georges Lema\^itre, 4,\\
B-1348 Louvain-la-Neuve, Belgium}


\begin{abstract}
We study decidability and complexity questions related to a continuous analogue
of the Skolem-Pisot problem concerning the zeros and nonnegativity of a linear
recurrent sequence. In particular, we show that the continuous version of the
nonnegativity problem is NP-hard in general and we show that the presence of a
zero is decidable for several subcases, including instances of depth two or less,
although the decidability in general is left open. The problems may also be
stated as reachability problems related to real zeros of \emph{exponential
polynomials} or solutions to initial value problems of linear differential
equations, which are interesting problems in their own right.
\end{abstract}

\end{frontmatter}

\section{Introduction}\label{intro}

Skolem's problem (also known in the literature as Pisot's problem) asks whether it
is algorithmically decidable if a given linear recurrent sequence (LRS) has a zero
or not. A LRS may be written in the form:
$$ u_k = a_{n-1}u_{k-1} + a_{n-2}u_{k-2} + \cdots + a_0u_{k-n},
$$
for $k \geq n$ where $u_0, u_1, \ldots, u_{n-1} \in \Z$ are
the initial inputs and $a_0, a_1, \ldots, a_{n-1} \in \Z$
are the recurrence coefficients, see also \cite{Ev03}. This forms the infinite sequence
$(u_k)_{k = 0}^{\infty} \subseteq \Z$. We may assume $a_0$ is nonzero,
otherwise a shorter and equivalent recurrence exists. Such a
recurrence sequence is said to be of \emph{depth} $n$.

For a linear recurrent sequence $u = (u_k)_{k = 0}^{\infty} \subseteq \Z$
the \emph{zero set} of $u$ is given by $Z(u) = \{i \in \N| u_i = 0\}$.
One of the first results concerning the zeros of LRS's was by
T. Skolem in \cite{Sk}, when he proved
that the zero set is semilinear (i.e., the union of finitely many periodic sets
and a finite set). This result was also later shown by K. Mahler \cite{Ma}
and C. Lech \cite{Le} and is now often referred to as the
\emph{Skolem-Mahler-Lech} theorem.
It is known that determining if $Z(u)$ is an infinite set is
decidable as was proven by Berstel and Mignotte \cite{BM}.

It was shown by N. Vereshchagin in 1985 that Skolem's problem
(i.e., the problem ``is the zero set of a LRS empty?'') is decidable
when the depth of the linear recurrent sequence is less than or equal to
four in \cite{V85}. It was also recently shown that Skolem's problem is
decidable for depth five in \cite{HHHK}, but the general decidability
status is open. It is also known that determining if a given linear
recurrent sequence has a zero is NP-hard, see \cite{BP02}.

Note that we may always encode a linear recurrent sequence of depth $n$
into an integral matrix $A \in \Z^{(n+1) \times (n+1)}$ such that
$u_k = A^k_{1,n+1}$ for $k \geq 1$. This follows since given the initial vector
$u = (u_0, u_1, \ldots, u_{n-1})^T$ and the recurrence coefficients,
$a_0, a_1, \ldots, a_{n-1}$, we first define matrix
$A' \in \Z^{n \times n}$:
$$
A' = \left( \begin{array}{ccccc}
 0 & 1 & 0 & \cdots & 0 \\ 0 & 0 & 1 & \cdots & 0 \\
 \vdots & \vdots & \ddots & \ddots & \vdots \\
 0 & 0 & 0 & \cdots & 1 \\ a_0 & a_1 & a_2 & \cdots & a_{n-1} \\
 \end{array}\right).
$$
Note that $(A')^ku = (u_k, u_{k+1}, \ldots, u_{k+n-1})$. Now we shall
extend this matrix by 1 dimension to give:
$$
A = \left( \begin{array}{cc} A' & A'v \\ \overline{0} & 0
\end{array}\right) \in \Z^{(n+1) \times (n+1)},
$$
where $\overline{0}$ is the zero vector of appropriate size. It is not
difficult to now see that $u_k = A^k_{1,n}$ for $k \geq 1$ as required.
Skolem's problem in this context is therefore to determine if the upper right
entry of a positive power of an integral matrix is zero. More generally, one can 
show that Skolem's problem is equivalent to the following problem: given a matrix 
$A\in \mb{F}^{n \times n}$ and two vectors $c,x_0 \in \mb{F}^{n},$ is there a 
nonnegative integer $t$ such that $c^TA^tx_0=0$? We add that a generalization of 
this problem where we may take any product of two integral matrices of dimension 
$10$ is known to be undecidable, see \cite{HH07}.

In this paper we shall consider a dynamical system whose updating
trajectory is given by $\frac{dx(t)}{dt} = Ax(t)$ where $A \in \R^{n \times n}$
and the initial point $x(0) \in \R^n$ is given. We shall
be interested in determining whether this trajectory ever reaches
a given hyperplane, thus the problem is equivalent to determining
if there exists $t \in \PR$ such that $c^T \exp(At)x(0) = 0$ where
$c \in \R^n$ defines the hyperplane. We consider this as
the Skolem-Pisot problem in continuous time. We show that for
instances of size two or less this problem is decidable.

We shall also show that determining if $c^T \exp(At)x(0)$ reaches
zero is \emph{computationally equivalent} to determining whether
a given real-valued exponential polynomial
$f(z) = \sum_{j = 1}^{m} P_j(z)\exp(\theta_jz)$, where each
$P_j$ is a polynomial, ever reaches zero for a \emph{positive real value}.
This is also equivalent to determining if the solution $y(t)$ of
an ordinary differential equation
$y^{(k)} + a_{k-1}y^{(k-1)} + \ldots + a_0y = 0$ with given initial
conditions $y^{(k-1)}(0), y^{(k-2)}(0), \ldots, y(0)$ ever reaches zero.

From 1920, P\'{o}lya and others characterized the asymptotic distribution
of complex zeros of exponential polynomials
\cite{Mo73, Po, Sch25, Ta25, vdP75, Wi17}. Upper bounds were also found
on the number of zeros in a finite region of the complex plane, using
the argument principle. Less is known about real zeros. Upper and lower
bounds on the number of zeros in a real interval are given in \cite{Vo76}.
A formula for the asymptotic density of real zeros for a restricted class
of exponential polynomials was found in \cite{Ka43}. Some observations on
the first sign change of a sum of cosines are collected in \cite{NS82}.
However, no criterion has been proposed to check the existence of a real
zero for a real exponential polynomial.

A related problem, determining whether a given linear recurrent sequence has only
nonnegative terms, the \emph{nonnegativity problem}, is decidable
for dimension $2$, see \cite{HHH}. The authors note that if the
nonnegativity problem is decidable in general, it implies Skolem's
problem is decidable. This follows since if $(u_k)_{k = 0}^{\infty}$
is recurrent, then so is $(u_k^2-1)_{k = 0}^{\infty}$.

We may note that using the linear recurrent sequence
$(u_k)_{k = 0}^{\infty}$ from the proof of NP-hardness of Skolem's
problem in \cite{BP02}, and converting it to the form
$(u_k^2-1)_{k = 0}^{\infty}$, allows one to easily derive the following
result:
\begin{thm}
It is NP-hard to decide if a given linear recurrent sequence is
nonnegative, i.e., the nonnegativity problem is NP-hard.
\end{thm}

This holds since if $(u_k)_{k = 0}^{\infty}$ is represented by a
matrix $\Z^{n \times n}$, then $(u_k^2-1)_{k = 0}^{\infty}$ may
be represented by a matrix $\Z^{(n^2+1) \times (n^2+1)}$ and thus
we have a polynomial time reduction. In this paper we show that the
nonnegativity problem in the continuous setting is also NP-hard.

Given a matrix $M \in \R^{n \times n}$ and vectors $u,v \in \R^n$,
the \emph{orbit problem} asks if there exists a power $k \in \N$ such
that $M^ku=v$. Thus it is a type of \emph{reachability problem}, see \cite{BT00}. This was shown
to be decidable even in polynomial time, see \cite{KL86}. The corresponding
version of this problem for continuous time asks whether for a given
$M \in \R^{n \times n}$ and vectors $a,b \in \R^n$ there exists some
$t \in \PR$ such that $\exp(Mt)a = b$. This problem was proved to be decidable
in \cite{Ha08}.

\section{Preliminaries}\label{PrelimSec}

Let $A \in \mb{F}^{n \times n}$ denote an $n \times n$ matrix
over the field $\mb{F}$ and $\sigma(A)$ the set of
eigenvalues of $A$. For a complex number $z \in \mathbb{C}$ we denote
by $\Re(z)$ the \emph{real} part of $z$ and by $\Im(z)$ the
\emph{imaginary} part of $z$. We use the notation $\PR$
to denote the nonnegative real numbers.

We shall denote an \emph{exponential polynomial}
$f:\C \to \C$ by a sum of the form:
$
f(z) = \sum_{j = 1}^{m} P_j(z)\exp(\theta_jz),
$
where $P_j \in \C[X]$ and $\theta_j \in \C$.

Given a matrix $A \in \C^{n \times n}$ we shall denote by the
\emph{dominant eigenvalues of $A$} the set of eigenvalues of $A$ with
maximum real part, i.e.,
$$
\{\theta \in \sigma(A)| \Re(\theta) \geq \Re(\theta'), \theta' \in \sigma(A)\}.
$$

We will later require the following theorem from Diophantine approximation \cite{Bk66}:

\begin{thm} \label{TheoBaker} (\textsc{Baker})
Let $\alpha_1, \ldots, \alpha_k, \beta_0, \ldots, \beta_k$
be algebraic numbers. Then the combination

$$\Lambda=\beta_0 + \sum_i \beta_i \ln \alpha_i$$
is either zero or satisfies $|\Lambda|>h^{-N}$, where $h$ is the largest height of 
$\beta_1, \ldots, \beta_k$, and $N$ is a computable constant depending only on 
$\ln \alpha_1, \ldots, \ln \alpha_k$ and the maximum degree of $\beta_0, \ldots, \beta_k$.
\end{thm}

Recall that for an algebraic number $\beta$ with minimal polynomial 
$$p(x) = \sum_{0\leq i \leq d} a_i x^i,$$ its degree is $d$ and its height is $\max{|a_i|}$.
We shall also use the following theorem regarding the transcendence degree of the field 
extension of algebraic numbers when considering their exponentials:

\begin{thm}\label{HLThm} (\textsc{Hermite-Lindemann}) - Let
$\alpha_j, \lambda_j \in \C$ for $0 \leq j \leq n-1$ be
algebraic numbers such that no $\alpha_j = 0$ and each $\lambda_j$
is distinct. Then:
$$
\sum_{j = 0}^{n-1} \alpha_j e^{\lambda_j} \neq 0.
$$
\end{thm}

The following theorem concerns
simultaneous Diophantine approximation of algebraic numbers
which are linearly independent over the rationals.

\begin{thm}\label{KronThm} (\textsc{Kronecker}, see \cite{Ca57})
Let $1,\lambda_1, \lambda_2, \ldots, \lambda_n \in \R$ be real
algebraic numbers which are linearly independent over $\Q$.
Then for any $\alpha \in \R^n$ and $\epsilon > 0$, there
exists $p \in \Z^n$ and $k \in \N$ such that
$|(k\lambda_i - \alpha_i - p_i)| < \epsilon$ for all
$1 \leq i \leq n$.
\end{thm}

\section{Skolem's Problem in Continuous Time}

We shall consider continuous time systems governed by the rule
$\frac{dx(t)}{dt} = Ax(t)$ where $A$ is a real matrix and $x(t)$ is a real vector
\footnote{We consider entries to be algebraic so that the input
to a problem has a finite description.}. We are interested in the
decidability of whether from an initial vector $x(0)$, we cross a
given hyperplane. We may consider this as a ``point-to-set''
reachability problem in a dynamical system, see \cite{BT00} for other examples.

Let $\frac{dx(t)}{dt} = Ax(t)$ where
$A \in \R^{n \times n}$ and $x(t) \in \R^n$.
Given the initial vector $x(0) \in \R^n$, then $x(t)$ is given by:
$$
x(t) = \exp(At) \cdot x(0) = \sum_{j = 0}^{\infty} \frac{t^j}{j!}A^j \cdot x(0).
$$

Given a vector $c \in \R^n$ defining a hyperplane, we would like
to determine if there exists some $t \in \PR$ such
that $c^Tx(t) = 0$. In other words, whether the flow of the
point $x(0)$ ever intersects the hyperplane. If such a
$t$ exists, we say that there exists a solution to the instance
\footnote{Note that, in the style of Skolem's problem, we shall be more
interested in determining whether any solution exists, rather
than trying to find an algebraic description of the solution.}.
An instance of \CSP{}{} therefore consists of the matrix
$A \in \R^{n \times n}$, the initial point $x(0) \in \R^n$
and the hyperplane vector $c \in \R^n$.

\subsection{Equivalent Formulations}

To analyze the behaviour of the system, we will convert
a given instance of \CSP{} into various forms which have different
properties but which are essentially equivalent to the original
problem.

Given such an instance, the following lemma shows that
the problem is equivalent to determining if the upper right entry
of the exponential of a matrix equals some constant real. A
similar construction is known in the discrete case as shown in
Section~\ref{intro}.

\begin{thm}\label{altProbVersion}
Given an instance of \CSP{} defined by $f(t) = c^T \exp(At)x(0)$
where $A \in \R^{n \times n}$ and $c,x(0) \in \R^n$. There exists
a polynomial-time computable matrix $B \in \R^{(n+2)\times(n+2)}$
such that $f(t) = \exp(Bt)_{1,n+2} + \lambda$, where
$\lambda = c^Tx(0) \in \R$ is constant.
\end{thm}

\begin{proof}
We are given the function $f(t) = c^T \exp(At)x(0)$.
Let $B \in \R^{(n+2)\times(n+2)}$ be given by:
$$ B \stackrel{\Delta}{=}
\left(\begin{array}{rrr}
0 & c^TA & c^TAx(0) \\ \bar{0} & A & Ax(0) \\ 0 & \bar{0}^T & 0 \\
\end{array}\right),
$$
where $\bar{0} = (0, 0, \ldots, 0)^T \in \R^{n}$, thus:
$$
\exp(B) =
\left(\begin{array}{rrr}
1 & c^T\exp(A)-c^T & c^T\exp(A)x(0) - \lambda \\
\bar{0} & \exp(A) & \exp(A)x(0)-x(0) \\ 0 & \bar{0}^T & 1 \\
\end{array}\right),
$$
where $\lambda = c^Tx(0)$ is constant.
This can be seen from the power series representation
$\exp(tB) = \sum_{j = 0}^{\infty} \frac{t^j}{j!}B^j$.
Therefore $f(t) = \exp(Bt)_{1,(n+2)} + \lambda$ and thus
an instance of \CSP{} can also be given by a single real matrix
$B$ and the problem of whether $f(t)$ reaches zero for $t \in \PR$
is equivalent to whether $\exp(Bt)_{1,(n+2)}$ ever equals $-\lambda.$
\end{proof}

\begin{thm}\label{expoSumThm}
The following problems are computationally equivalent with
polynomial time reductions (where all \emph{parameters} are algebraic numbers):
\begin{itemize}
\item[(i)] Does there exist a solution to a given instance of \CSP{}?
\item[(ii)] Determine if a real-valued exponential polynomial:
$$
f(t) = \sum_{j=1}^mP_j(t)e^{\theta_jt},
$$
has a nonnegative real zero (where $\theta_j \in \C$ and
$P_j \in \C[X]$).
\item[(iii)] Determine if a function of the form:
$$
f(t) = \sum_{j=1}^m e^{r_j t}(P_{1,j}(t)\cos(\lambda_jt) + P_{2,j}(t)\sin(\lambda_jt))
$$
has a nonnegative real zero (where $r_j, \lambda_j \in \R$ and
$P_{i,j} \in \R[X]$).
\item[(iv)] Determining whether the solution $y(t)$ to an ordinary
differential equation $y^{(k)} + a_{k-1}y^{(k-1)} + \ldots + a_0y = 0$
with the given initial conditions $y^{(k-1)}(0), y^{(k-1)}(0), \ldots, y(0)$
reaches zero for a nonnegative real $t$.
\end{itemize}
\end{thm}
\begin{proof}
$(i) \Rightarrow (ii)$: Let $J \in \C^{n \times n}$ be the Jordan
matrix for $A$, thus we may write $A = PJP^{-1}$ for some
$P \in GL(n, \C)$.\footnote{These can be effectively found since
we only need \emph{algebraic descriptions} of the Jordan normal form $J$
and the similarity matrix $P$.} Since $\exp(PJP^{-1}) = P\exp(J)P^{-1}$, we
can ask the equivalent problem, does there exist a time
$t \geq 0$ at which:
\begin{eqnarray}
c^T y(t) & = & c^T \exp(tA)y(0) \nonumber \\
 & = & u^T \exp(tJ) v \nonumber = 0, \nonumber
\end{eqnarray}
where $u,v \in \C^n$ are defined by $u^T = c^TP$ and $v = P^{-1}y(0)$?

Let $J = J_1 \oplus J_2 \oplus \ldots \oplus J_m$ be a decomposition of
$J$ into a direct sum of Jordan blocks with $J_i \in \C^{n_i \times n_i}$
and $\sum_{i = 1}^m n_i = n$. Each Jordan block may be written
$J_i = \theta_i I_{n_i} + M_i$ where $\theta_i \in \mathbb{C}$ is the
associated eigenvalue, $I_{n_i} \in \Z^{n_i \times n_i}$ is the
identity matrix and $M_i \in \Z^{n_i \times n_i}$ has $1$ on the
super-diagonal and $0$ elsewhere.

For $1 \leq i \leq m$, we see that $\theta_i I_{n_i}$
and $M_i$ commute and therefore $\exp(tJ_i) = \exp(t\theta_i I_{n_i})\exp(tM_i)$.
The value of $\exp(t\theta_i I_{n_i})$ is $e^{t\theta_i}I_{n_i}$. Let
$\exp(tM_i) = [m_{jk}] \in \mathbb{Q}^{n \times n}$, then
\begin{equation}\label{polyDef}
m_{jk} = \left\{\begin{array}{ll} \frac{t^{(k-j)}}{(k-j)!} &; \textrm{ if } j \leq k\\
0 & ; \textrm{ otherwise } \end{array} \right.
\end{equation}

Therefore we may convert our problem equivalently into deciding
whether there exists a $t \in \PR$ such that $f(t) = 0$ where
$f: \R \to \C$ is defined by:

\begin{equation}\label{expEqn}
f(t) = u^T  \exp(Jt) v = \sum_{j=1}^mP_j(t)e^{\theta_jt},
\end{equation}
and $P_j \in \C[X]$ are polynomials and whose
degree depends upon the size of the corresponding Jordan block and
$\theta_j \in \C$. The polynomials $P_j$ can be derived from
Equation~(\ref{polyDef}). Note that each of these steps is effective
and can be computed in polynomial time for algebraic entries of the
initial matrix $A$.

$(ii) \Rightarrow (iii)$:
This results from Euler's formula for the complex exponential
and the fact that $f(t)$ is a real valued function.

$(iii) \Rightarrow (iv)$: Functions of the type
$$
f(t) = \sum_{j=1}^m e^{r_j t}(P_{1,j}(t)\cos(\lambda_jt) + P_{2,j}(t)\sin(\lambda_jt))
$$
where $r_j, \lambda_j \in \R$ are fixed and $P_{k,j}$ are arbitrary
real polynomials of degree $\leq d_j$ form a real vector space
of dimension $k = 2 \sum_{j = 1}^m(d_j+1).$ This vector space is
closed under differentiation. Hence the first $(k+1)$
derivatives of $f$ are related by
$f^{(k)} + a_{k-1}f^{(k-1)} + \ldots + a_0 f = 0$ where each $a_j$
can be found in polynomial time. By Cauchy's theorem for ordinary
differential equations, a function $f$ is completely determined by the given
relation and the initial conditions
$f^{(k-1)}(0), f^{(k-2)}(0), \ldots, f(0)$.

$(iv) \Rightarrow (i)$: The characteristic equation of the
linear homogeneous differential equation is given by
$z^k + z^{k-1}a_{k-1} + \ldots + a_0 = 0$. It is well known that
we can form the companion matrix of the equation in order to
convert the problem into an instance of \CSP{}. The initial
values are then present in the initial vector $x(0)$.
\end{proof}

\begin{lem}\label{lambdaLem}
Let $A \in \R^{n \times n}$ and $c, x(0) \in \R^n$ form an instance of \CSP{}. 
For any $\lambda \in\mb{C}$ we may form a system
$f_\lambda(t) = u^T \exp(t(A + \lambda I)) v$ where
$u,v \in \mb{C}^n$, $\sigma(A + \lambda I) = \sigma(A)+\lambda$ and $f(t) = 0$
if and only if $f_\lambda(t) = 0$.
\end{lem}

\begin{proof}
Let $\lambda \in \mathbb{C}$ and define $y(t) = e^{\lambda t}x(t)$, thus:
\begin{eqnarray} \frac{dy(t)}{dt} & =
	&\lambda e^{\lambda t} x(t) + e^{\lambda t}\frac{dx(t)}{dt}\nonumber\\
	& = &e^{\lambda t} (\lambda I + A)x(t) \nonumber\\
	& = &(\lambda I + A)y(t) \nonumber
\end{eqnarray}

Define $A_{\lambda} = \lambda I + A$, thus:
$$ y(t) = \exp(tA_{\lambda}) y(0).
$$

Note that there exists $t \geq 0$ such that $c^Tx(t) = 0$ if and
only if $c^Ty(t) = 0$.
\end{proof}

As an example, which will be useful later, let us set
$\lambda = -\textrm{max}\{\Re(\theta)| \theta \in \sigma(A)\}$,
so that all eigenvalues are shifted to the left complex half-plane or
the imaginary axis. This means that we have, in effect, split the set of
eigenvalues into two sets, one which decays exponentially with time and
one which consists of purely imaginary values.

We now remark that any nontrivial solution to the problem
will in fact be \emph{transcendental}.

\begin{thm}\label{transSol}
Given an instance of \CSP{}, all solutions, if any exist,
are transcendental unless the polynomials $P_j(t)$ share a common positive real root.
\end{thm}

\begin{proof}
The corresponding exponential polynomial
formed as in Theorem~\ref{expoSumThm} will be in the form:
$$
f(t) = \sum_{j=1}^mP_j(t)e^{\theta_jt} = 0.
$$
We may assume no $P_j \in \C[X]$ is zero otherwise simply
remove it from the sum and that each $\theta_j$ is distinct,
otherwise group them together. Thus, according to
Theorem~\ref{HLThm} (the Hermite-Lindemann theorem), this
exponential polynomial only has solutions for transcendental
times $t$ where $t \in \PR$.
\end{proof}

\section{Decidable Cases}

We shall now investigate some classes of instances for which \CSP{}
is decidable.

\begin{thm}
The \CSP{} for depth $2$ is decidable.
\end{thm}

\begin{proof}
Assume we have an instance of \CSP{} given by
$f(t) = (c_1, c_2) \exp(At)(x_1,x_2)^T$ with
$A \in \mathbb{R}^{2 \times 2}$. Let
$S \in GL(\C,2)$ put $A$ into Jordan canonical form.
We can rewrite $f(t)=(\alpha_1, \alpha_2) \exp(Jt)(\beta_1,\beta_2)^T,$ where $J = S^{-1}AS$
is a Jordan matrix.

If $A$ has one eigenvalue $\theta$, with algebraic multiplicity
$2$, then $\theta \in \R$. If $\theta$ has geometric multiplicity $1$ then by
Theorem~\ref{expoSumThm} we must solve an equation of the form $(1 + xt)ye^{t\theta}$ where
$x, y \in \R$, thus the instance has a solution if and only if
$-\frac{1}{x} \in \PR$. If $\theta$ has geometric multiplicity $2$ then
we must solve $e^{\theta t}(\alpha_1\beta_1 + \alpha_2\beta_2) = 0$
which has a solution if and only if $(\alpha_1\beta_1 + \alpha_2\beta_2) = 0$.

Otherwise, $J$ is diagonal and we must determine if there
exists a $t \in \PR$ such that
$e^{t\theta_1} + \alpha e^{t\theta_2} = 0$ for $\alpha \in \R$.
Either $\theta_1,\theta_2 \in\R$ or $\theta_1 =\overline{\theta_2}\in \C$.

If $\theta_1,\theta_2 \in\R$ assume without loss of generality
that $\theta_1 < \theta_2$ and we have
$g(t) = e^{t\theta_1} + \alpha e^{t\theta_2}$ thus, by taking logarithms,
$t = \frac{\textrm{ln}(-\alpha)}{\theta_1 - \theta_2}$ is a solution of $g(t) = 0$
and thus there exists a solution if and only if
$\frac{\textrm{ln}(-\alpha)}{\theta_1 - \theta_2} \in \PR$.

In the other case $\theta_1 = \overline{\theta_2} \in \C$.
Since we may therefore shift the real part as allowed by \ref{lambdaLem}, assume
that $\theta_1, \theta_2 \in i\R$. At time
$t = \frac{\pi}{2\Im(\theta_1)}$ we have
$$\begin{array}{rl} e^{t\theta_1} + \alpha e^{t\overline{\theta_2}} & =
e^{\Im(\theta_1)it} + \alpha e^{-\Im(\theta_1)it} \\ &
= \cos(\frac{\pi}{2}) + \alpha \cos(-\frac{\pi}{2}) = 0
\end{array}
$$
which is a solution, thus we are done.
\end{proof}

The following theorem shows that the class of instances where all elements
of the input are nonnegative reals in the continuous
setting is trivially decidable in polynomial time, whereas in the discrete time
case, the problem is NP-hard, as shown in \cite{BP02}. In fact, using
Lemma~\ref{lambdaLem}, we see that in the continuous setting the Skolem-Pisot
problem is polynomially decidable even where the matrix given is a \emph{Metzler matrix},
meaning only off-diagonal elements need be nonnegative.

\begin{thm}
For an instance of \CSP{} given by
$A \in \R^{n \times n}$ and $c, x(0) \in \PR^n$ where $A$ is a Metzler matrix
(thus all off-diagonal elements are nonnegative) and
$f(t) = c^T\exp(At)x(0),$ then we may decide if there exists
a solution in polynomial time.
\end{thm}

\begin{proof}
Let $\lambda$ be the minimal diagonal element of $A$. If $\lambda < 0$ then
by Lemma~\ref{lambdaLem}, we may form an equivalent instance
$A' = A + \lambda I$ where $A' \in \PR^{n \times n}$. Thus assume without loss
of generality that $A$ is a nonnegative matrix and $c,x(0)$ are nonnegative vectors.

Note that $\exp(t_2A) > \exp(t_1A)$ for any $t_2 > t_1 \in \PR$
which is a consequence of the power series representation of
$\exp(At) = \sum_{j = 0}^{\infty} \frac{t^j}{j!}A^j$ and the fact
that $A \in \PR^{n \times n}$. We see that $f(0) = c^Tx(0) \in \PR$.
Now, if $f(0) = 0$ then this is a solution, otherwise, since the
matrix exponential increases monotonically componentwise with time
for a nonnegative matrix, there exists no solution.
\end{proof}

In some special cases, some eigenvalues of $A$ do not influence the function $f(t).$
This is easily seen when $A$ is put in its Jordan form $J=P^{-1}AP$: \begin{equation}\label{f(t)-jordan}f(t)=u^T \exp(tJ) v\end{equation}
where $u,v \in \C^n$ are defined by $u^T = c^TP$ and $v = P^{-1}x(0).$  Obviously, if the
entries of $c$ or $x(0)$ corresponding to a particular Jordan block are zero, this block does not
play any role and one may remove it without changing the function $f(t)$.
More generally, from Equation~(\ref{f(t)-jordan}) it is easy, as shown in Theorem~\ref{expoSumThm}, to write the function $f$ as follows:
 \begin{equation}\label{eq-f-red}
f(t) = \sum_{j=1}^mP_j(t)e^{\theta_jt},
 \end{equation} 
where the $\theta_j$ are
the distinct eigenvalues of $A$ and  the $P_j$ are  complex polynomials.   
If no polynomial $P_j(t)$ is identically zero, then we say that the triple $(A,c,x(0))$ is 
\emph{reduced}. If some of the $P_j$ are zero, we can remove the corresponding terms from 
Equation~(\ref{eq-f-red}), since it does not change the value of $f.$

Theorem \ref{expoSumThm} shows how to build an equivalent instance of the form $(A',c',x'(0))$ from 
an instance of the form of Equation~(\ref{eq-f-red}). One would then obtain a reduced instance of the 
\CSP{}. This can be done in a preprocessing phase.

\begin{thm}\label{eigenThm}
Let $\frac{dx(t)}{dt} = A x(t)$ for $A \in \mathbb{R}^{n \times n}$ and
$x(t) \in \mathbb{R}^n$ define an instance of \CSP{} given by
$f(t) = c^T \exp(At)x(0) = 0$. If $(A,c,x(0))$ is reduced and none of the dominant eigenvalues of $A$ are real
then the problem is decidable.
\end{thm}

\begin{proof}
By Lemma~\ref{lambdaLem}, let us assume all eigenvalues
have real part less than or equal to $0$. Then, using
Theorem~\ref{expoSumThm}, we may consider the system as being
represented by
$$
f(t) = \sum_{j=1}^mP_j(t)e^{\theta_jt}.
$$
We may split this exponential polynomial in two (reordering as necessary)
and write $f(t)=f_1(t)+f_2(t)$, where
\begin{eqnarray}f_1(t)& =&\sum_{j=1}^k P_j(t) \exp(i\lambda_j t) \nonumber\\
& =& \sum_{j=1}^k \left(P_{1,j}(t) \cos{(\lambda_j t)} +
P_{2,j}(t) \sin(\lambda_j t)\right) \label{polyEx1}
\end{eqnarray}
are those terms with $0$ real part and $f_2(t)$ is the summation of the remaining terms.
Equation~(\ref{polyEx1}) follows from Theorem~\ref{expoSumThm} since $f_1(t)$ is
real valued.
Since $(A,c,x(0))$ is reduced, $f_1(t)$ is not identically zero.
Note that $f_2(t)$ tends to zero exponentially fast
as $t$ increases.

For a polynomial $P$ of degree $n$ we may use
Cauchy's bound on the maximum modulus of any polynomial root to determine
that for any root $z \in \C$ of $P(x) = a_{n}x^{n} + \ldots + a_1x + a_0$
we have that:
$$
|z| \leq 1 + \frac{\textrm{max}\{|a_0|,|a_1|, \ldots, |a_{n_j-1}|\}}{|a_{n}|},
$$
as is easy to prove.
Thus define $T \in \PR$ to be strictly greater than this
maximum bound for any $P_{1,j}$ or $P_{2,j}$ in Equation~(\ref{polyEx1})
for $1 \leq j \leq m$ and thus for all $t \geq T$, the sign
of $P_{1,j}(t)$ and $P_{2,j}(t)$ for each $1 \leq j \leq m$ is fixed.

For each $1 \leq j \leq k$ there exists $t_{j, 1},t_{j, 2} > T$ such that
$$P_{1,j}(t_{j,1})\cos(\lambda_jt_{j,1})+P_{2,j}(t_{j,1})\sin(\lambda_jt_{j,1})>0,$$
$$P_{1,j}(t_{j,2})\cos(\lambda_jt_{j,2})+P_{2,j}(t_{j,2})\sin(\lambda_jt_{j,2})<0.$$

Each $\lambda_j$ is distinct thus we have enough freedom in the choice
of these times so that there exists $t_1, t_2 >T$ such that

$$\begin{array}{c}
P_{1,j}(t_1)\cos(\lambda_jt_1)+P_{2,j}(t_1)\sin(\lambda_jt_1)>0 \\
P_{1,j}(t_2)\cos(\lambda_jt_2)+P_{2,j}(t_2)\sin(\lambda_jt_2)<0 \\ \end{array}
; \,\, 1 \leq j \leq k.
$$
We now see that $f_1(t_1)$ is positive and $f_1(t_2)$ is negative thus
there exists a solution since $f_2(t)$ decays exponentially
fast and there exists an infinite number of solution times.
\end{proof}

We now use Theorem~\ref{TheoBaker} (Baker's theorem) to provide bounds on sums of exponentials polynomials.
We first start with some lemmata.

\begin{lem}\label{epsilnLemma2}
Let $\omega_1$ and $\omega_2$ be different algebraic numbers, linearly independent over $\Q$, and 
$e^{i \phi_1}, e^{i \phi_2}$ be algebraic numbers on the unit circle.
There exist effective constants $C, N, T >0$ such that at any time instant $t>T$, either
$1-\cos (\omega_1 t + \phi_1) > C/t^N$ or $1-\cos (\omega_2 t + \phi_2) > C/t^N$.
\end{lem}

\begin{proof}
We prove that there are $C,N$ such that $|\omega_1 t + \phi_1 - 2k \pi| > C/t^N$ for all integers $k$ 
or $|\omega_2 t + \phi_2 - 2k \pi| > C/t^N$ for all integers $k$.

Indeed, suppose that for some $t, k, l$ and $\epsilon >0$, both $|\omega_1 t + \phi_1 - 2k \pi| \leq \epsilon$ 
and $|\omega_2 t + \phi_2 - 2l \pi| \leq \epsilon$. Then $|t + \phi_1/\omega_1 - 2k \pi/\omega_1| \leq \epsilon/|\omega_1|$ and
 $|t + \phi_2/\omega_2  - 2l \pi/\omega_2 | \leq \epsilon/|\omega_2|$. By difference we find
 $|\phi_1/\omega_1 - \phi_2/\omega_2 + 2l \pi /\omega_2- 2k \pi /\omega_1| < \epsilon(\frac{1}{|\omega_1|}+\frac{1}{|\omega_2|})$. 
Let us introduce $\omega=(\frac{1}{|\omega_1|}+\frac{1}{|\omega_2|})^{-1}$. Then 
$|k \frac{\omega}{i \omega_1}2 \pi i - l \frac{\omega}{i \omega_2} 2 \pi i + \frac{\omega}{i \omega_1} i \phi_1 - \frac{\omega}{i \omega_2} i \phi_2 | < \epsilon$.

 Observing that $2 \pi i, i\phi_1,i\phi_2$ are logarithms of algebraic numbers, we apply Baker's theorem.
Note that the height of $k \alpha$, for any algebraic number  $\alpha$ of degree $d$, is at most $|k|^d$ times
 the height of $\alpha$, from the definition of height.
Thus, $\epsilon > \max{(C_1 |k|^{d}, C_2|l|^{d}, C_0)}^{-N_0}$,
for some $C_0, C_1, C_2, N_0$ not depending on $k,l$. It is also clear that, given $t$,
 $|\omega_1 t + \phi_1 - 2k \pi|$ is closest to zero for some $k < C_3 t$, and similarly for
 $|\omega_2 t + \phi_2 - 2l \pi|$.

This proves that there exists $C',N_0, T$ such that for every $t > T$, either 
$|\omega_1 t + \phi_1 - 2k \pi| > C' t^{-N_0}$ for all $k$, or $|\omega_2 t + \phi_2 - 2l \pi| > C' t^{-N_0}$ 
for all $l$. Since $1- \cos (\alpha+ 2k \pi) >\alpha^2/3$ for $\alpha$ small enough and some $k$ (from 
Taylor approximation), the claim follows.
\end{proof}

We say that a property \emph{$T$-eventually} holds for  a function $g: \R_{\geq 0} \to \R$ if it 
holds for all time instants $t \geq T$. For instance, $g$ is \emph{eventually positive} if 
there is a threshold $T$ such that $g(t)> 0$ for all $t \geq T$.
Clearly, if $f$ is the solution of a linear differential equation, then it has finitely many 
zeros if and only if it is eventually positive or eventually negative.

We say that $g_1$ is \emph{$(T,r)$-exponentially dominated} by $g_2$ if
$|g_1(t)| < e^{-rt} |g_2(t)|$ for $r>0$ and all $t \geq T$.

\begin{lem} \label{lemm_g}
Let us consider $T$-eventually nonzero continuous functions $g_1, \ldots, g_k$ and the function
$f(t)=g_0(t) + \sum_{j=1}^k g_j(t) \cos (\omega_j t + \phi_j)$, where $\omega_1, \omega_2$ are linearly 
independent positive algebraic numbers, $g_0$ is $(T,r)$-exponentially dominated
by $g_1$ and $g_2$,  and $\phi_1, \phi_2$ are angles such that $e^{i\phi_1}, e^{i\phi_2}$ are algebraic.
Then the following $T'$-eventually holds, for some $T'$:
$$   - \sum_{j=1}^k |g_j| <  f < \sum_{j=1}^k |g_j|.$$
Moreover, such $T'$ can be computed as a function of $T, r, \omega_1, \omega_2, \phi_1, \phi_2$.

If all $\omega_j$ are linearly independent over the rationals, then for  any $\epsilon > 0$, 
there exist arbitrarily large times $t$ such that
$$f(t) > (1-\epsilon)\sum_{j=1}^k |g_j(t)|$$
and arbitrarily large times  $t$ such that
$$f(t) < -(1-\epsilon)\sum_{j=1}^k |g_j(t)|.$$
\end{lem}

\begin{proof}
It is obvious that $-|g_0| - \sum_{j=1}^k |g_j|<f < |g_0| + \sum_{j=1}^k |g_j|$.

To prove that we can get rid of $g_0$, we exploit the fact that
the cosines $\cos (\omega_1 t + \phi_1)$ and $\cos (\omega_2 t + \phi_2)$ never
 take the value $\pm 1$ exactly at the same time, except possibly once; this is
 a consequence of the linear independence of $\omega_1$ and $\omega_2$. Due to Lemma~\ref{epsilnLemma2}, 
 one of the cosines is $C t^{-N}$ away from $1$, for some $C,N$
 and all large enough times.
Then for all large enough times $t$,
either $|g_0 + g_1 \cos (\omega_1 t + \phi_1)+g_2 \cos (\omega_2 t + \phi_2)| < |g_1| (1-C t^{-N}) + |g_2|$ or 
$|g_0 + g_1 \cos (\omega_1 t + \phi_1)+g_2 \cos (\omega_2 t + \phi_2)| < |g_1|  + |g_2|(1-C t^{-N})$. In any
case, since $g_0$ is exponentially dominated by both $g_1$ and $g_2$, we
have $ |g_0 + g_1 \cos (\omega_1 t + \phi_1)+g_2 \cos (\omega_2 t + \phi_2)| < |g_1| + |g_2|,$ for
some $T'$, computable as a function of $T,r, C, N.$ Adding all the terms $g_j \cos (\omega_j t + \phi_j)$ proves
the first claim of the theorem.

We now prove the second claim. From Kronecker's theorem and the linear independence of frequencies,
we have that the set $\Gamma = \{(\cos (\omega_j t + \phi_j))_{1 \leq j \leq  k}| t \geq 0\}$ is dense in 
$[-1,1]^k$. Hence, $\Gamma$ will  approach all the vertices of  $[-1,1]^k$ by
less than any $\epsilon > 0$ for arbitrarily large times. For those times such that for all $j$, $\cos (\omega_j t + \phi_j)$
is close within $\epsilon/2$ to $\textrm{sign}\,\,\, g_j,$ and  $|g_0(t)/g_1|<\epsilon/2$, 
the first part of the second claim holds. The second part is similar.
\end{proof}

We now prove the main theorem of this section, which says that in some circumstances, one can 
reduce the search for a solution to an instance of \CSP{} to a \emph{finite} time interval.
Recall that an eigenvalue is \emph{nondefective} if its
algebraic and geometric multiplicities coincide. The \emph{frequency}
of an eigenvalue is the absolute value of the imaginary part.
Recall that for a real matrix, complex eigenvalues come in conjugate pairs, determining one equal frequency.
\begin{thm}
Given an instance of \CSP{} where all dominant eigenvalues are nondefective, at least four in number 
and such that the set of their distinct nonzero frequencies is linearly independent over the rationals.

Then
\begin{itemize}
\item The existence of infinitely many solutions is decidable;
\item If there are finitely many solutions, then those solutions are
in $[0,T]$, where $T$ is  computable. 
\end{itemize}
\end{thm}

Note that multiple dominant eigenvalues are allowed.

\begin{proof}
As allowed by Lemma~\ref{lambdaLem}, we can suppose without
loss of generality that the dominant eigenvalues are on the imaginary axis.

Then we are looking for real zeros of a function
$f(t)=\gamma_0 + f_1(t)+f_2(t)$, where $\gamma_0$ is the contribution of the dominant zero eigenvalue (if any),
\begin{eqnarray}f_1(t)& =&\sum_{j=1}^k z_j \exp(i\lambda_j t) \nonumber\\
& =& \sum_{j=1}^k \alpha_j \cos{(\lambda_j t)}+\beta_j \sin{(\lambda_j t)}  \nonumber
\end{eqnarray}
collects the dominant terms corresponding to dominant complex eigenvalues $\theta_j= i \lambda_j$ and $f_2(t)$ is exponentially
decreasing. By elementary trigonometric manipulations, $f_1$ can be converted into

\begin{eqnarray}f_1(t)& =& \sum_{j=1}^k \gamma_j \cos{(\lambda_j t + \phi_j)}, \nonumber
\end{eqnarray}
for some $\phi_j$ such that $\exp{(i \phi_j)}$ is algebraic. Hence $f_1$ is a linear combination of
shifted cosines.

Since there are at least four distinct dominant eigenvalues, $f_1$ contains at least two different frequencies.
We apply Lemma \ref{lemm_g}, with $g_i=\gamma_{i}, g_{0}=f_2$ to obtain that the following eventually holds:
$$-\sum_{j=1}^k |\gamma_j| < f-\gamma_0 < \sum_{j=1}^k |\gamma_j|.$$

Moreover, the same lemma tells us that for any $\epsilon > 0$ there are arbitrarily large times $t$ such that:
$$-(1-\epsilon)\sum_{j=1}^k |\gamma_j| \geq f(t)-\gamma_0$$ and  arbitrarily large times $t$ such that:
$$f(t)-\gamma_0 \geq (1-\epsilon)\sum_{j=1}^k |\gamma_j|.$$

As mentioned above, $f$ has finitely many zeros if and only if $f$ is eventually positive or 
eventually negative. It results from the above that $f  < 0$ ($T$-eventually, for some $T$)  
if and only if $\gamma_0 + \sum_{j=1}^k |\gamma_j| \leq 0$. Moreover, when $f < 0$ ($T$-eventually, 
for some $T$), such a $T$ can be be computed. A similar argument holds for $f>0$.
This proves the claim.
\end{proof}

Note that in discrete time, checking the existence of a zero in a finite time interval is a 
trivial task, while in continuous time we do not know how to decide the existence of a zero
between time $0$ and $T$.

\section{NP-Hardness of Nonnegativity Problem}
We now prove the continuous version of Blondel-Portier's result \cite{BP02}.

\begin{thm}
The nonnegativity problem for instances of \CSP{} given by a skew-symmetric
matrix is NP-hard and decidable in exponential time. In particular, the
general nonnegativity problem is NP-hard.
\end{thm}

\begin{proof}
A skew symmetric matrix has only imaginary eigenvalues and Jordan
blocks of size one. By Theorem~\ref{expoSumThm} we must find nonnegative
real zeros of a function of the form
$$f(t)=\sum_i \alpha_i \cos(\lambda_i t)+ \beta_i \sin(\lambda_i t).$$

We can, in polynomial time, find a basis $\xi_1, \ldots, \xi_m$ over the
rationals for the family $\lambda_1,\ldots, \lambda_k$, such that every
$\lambda_i$ is an integral combination of $\xi_1, \ldots, \xi_m$. For
every $\xi_i$ we introduce two variables $x_i = \cos(\xi_i t)$ and
$y_i= \sin(\xi_i t)$, which satisfy  $x_i^2+y_i^2=1$. Hence $f(t)$ is a
polynomial $P$ in $x_i, y_i$ (by elementary trigonometry). From
Theorem~\ref{KronThm} (Kronecker's theorem), the trajectory
$(\xi_1 t, \ldots, \xi_k t)$ is dense in $[0, 2\pi]^k$, from which
$\overline{\{f(t)|t \in \R\}}=\{P(x_1,y_1, \ldots, x_m, y_m)|x_i,y_i \in \R \}$
follows. Hence, $f$ is nonnegative if and only if $P$ is,
when taken over the set $\{x_1, y_1, \ldots, x_m, y_m | x_j,y_j \in\R \textrm{ and }
x_j^2 + y_j^2 = 1 \textrm{ for } 1\leq j \leq m\}$. This problem
is solvable in time exponential in $m$ by Tarski's procedure
(see for example \cite{BPR96}).

Suppose we are given a polynomial
$P(x_1, \ldots, x_k)$. We write $x_i=\cos(\xi_i t)$ for every $1 \leq i \leq k$.
Every monomial of $P$ can therefore be written as a linear combination of
cosines by elementary trigonometry. For instance,
$x_1x_2=\cos \xi_1 t \cos \xi_2 t = \frac{\cos(\xi_1 -\xi_2) t
+\cos (\xi_1 +\xi_2) t }{2}$, and so on. In this way, the polynomial
$P(x_1, \ldots, x_k)$  can be written as a function
$f(t)= \sum_i \alpha_i \cos (\lambda_i t)$, such that
$\overline{\{f(t) | t \in \R\}}=\{P(x_1, \ldots, x_k) |  x_i \in [-1,1] \}$.
Hence $f$ is nonnegative if and only if $P$ is nonnegative
on $[-1,1]^k$. Since checking the nonnegativity of a polynomial on
$[-1,1]^k$ is NP-hard (which is easily proved via an encoding of the
3-SAT problem, see, e.g., \cite{GJ79}), then the nonnegativity problem for instances of \CSP{}
is also NP-hard.
\end{proof}

Note that physical linear systems that preserve energy can often be
modelled by differential equations with a skew-symmetric matrix, because these
are precisely, up to a change of variables, the systems for which the energy $1/2 x^T x$
(where $x$ is the state) is constant along the trajectories,
(see, e.g., \cite{Wi72}). This case is therefore particularly relevant.

\section{Conclusion}

In studying this problem, we are not so much interested in
exactly describing the solutions to the problem, as determining the
\emph{existence} of solutions. For example, if we have algebraic times
$t_1, t_2 \in \PR$ with $t_1 < t_2$ such that $f(t_1)$ and $f(t_2)$ have
different signs then there exists $t \in [t_1, t_2]$ such that $f(t) = 0$
by the intermediate value theorem.

The main problem encountered in solving \CSP{} however appears to be that
$f(t)$ can reach $0$ tangentially,
i.e. we may have a solution $f(t) = 0$ where there exists
$\varepsilon > 0$ such that $f(\tau) \geq 0$ for all
$\tau \in [t - \varepsilon, t + \varepsilon]$. Since, by
Lemma~\ref{transSol}, the solution will, for non trivial cases,
be transcendental, it is difficult to determine when such a
situation arises. Indeed, given a real valued exponential
polynomial, if we take its square then it is positive real
valued and reaches zero tangentially if and only if the
first exponential polynomial had a zero.

We have therefore attempted to show several instances in
which the problem is decidable but the general problem remains open.
The equivalent problem of determining if an exponential polynomial
has real zeros seems equally interesting.
It is surprising that the problem is open even for a
finite time interval. Solving Skolem's problem in the discrete case
for finite time is obviously decidable since we can simply compute
the values in the interval.

\begin{open}
Is \textsc{Bounded Continuous Skolem's Problem} decidable? I.e. Given a
fixed $T \in \PR$, and an instance of \CSP{},
$f(t) = c^T \exp(At)x(0)$, does there exist $t \leq T$ such that $f(t) = 0$?
\end{open}

We also showed that the nonnegativity problem is NP-hard in the continuous
case. It is not clear if a similar technique can be used to show that \CSP{}
is also NP-hard. In the discrete Skolem's problem it turns out that determining the
nonnegativity and positivity of a linear recurrent sequence are
equivalent in terms of complexity, however this is not clear in
the continuous case.

\begin{open}
Are \CSP{} and the continuous nonnegativity problem
computationally equivalent?
\end{open}

\section{Acknowledgements}
This article presents research results of the Belgian Programme on Interuniversity
Attraction Poles, initiated by the Belgian Federal Science Policy Office.
This research has been also supported by the ARC (Concerted Research Action)
``Large Graphs and Networks'', of the French Community of Belgium. The scientific
responsibility rests with the authors. J.-C. Delvenne and R. Jungers hold
FNRS fellowships (Belgian Fund for Scientific Research).

We would like to thank Alexandre Megretski for a useful discussion on this problem.
We also would like to thank the two reviewers for their careful reading of this manuscript and
helpful comments and suggestions.


\begin{thebibliography}{00}

\bibitem{Apos}  T. Apostol, \emph{Liouville's Approximation Theorem}, in: Modular Functions and 
Dirichlet Series in Number Theory, 2nd ed. New York: Springer-Verlag, pp. 146-148, 1997.

\bibitem{BBC}  L. Babai, R. Beals, J. Cai, G. Ivanyos and
E. M. Luks, \emph{Multiplicative Equations over Commuting Matrices},
Proc. 7th ACM-SIAM Symp. on Discrete Algorithms , pp. 498-507, 1996.

\bibitem{Bk66} A. Baker, \emph{Linear Forms in the Logarithms of Algebraic Numbers} I-IV, 
Mathematika 13 (1966), 204-216; 14 (1967), 102-107 et 220-224; 15 (1968), 204-216.

\bibitem{BPR96} S. Basu, R. Pollack, M. Roy, \emph{On the Combinatorial
and Algebraic Complexity of Quantifier Elimination}. J. ACM 43, 6,
1002-1045, 1996.

\bibitem{BM} J. Berstel, M. Mignotte, \emph{Deux Propri\'{e}t\'{e}s
D\'{e}cidables des Suites R\'{e}currentes Lin\'{e}aires}, Bull. Soc. Math.
France, 104, 175-184, 1976.

\bibitem{BP02} V. Blondel, N. Portier, \emph{The Presence of a Zero in an
Integer Linear Recurrent Sequence is NP-hard to Decide}, Linear Algebra and
its Applications, Elsevier, 351-352, pp. 91-98, 2002.

\bibitem{BT00} V. D. Blondel, J. N. Tsitsiklis, \emph{A Survey of Computational
Complexity Results in Systems and Control}, Automatica, Elsevier, 36:9, pp.
1249-1274, 2000.

\bibitem{Ca57} J. Cassels, \emph{An Introduction to Diophantine
Approximation}, Cambridge Univ. Press, 1957.

\bibitem{Ev03} G. Everest, A. van der Poorten, I. Shparlinski, T. Ward,
\emph{Recurrence Sequences}, AMS Surveys and Monographs, Volume 104, ISBN 0-8218-3387-1, 2003.

\bibitem{GJ79} M.R. Garey, D.S. Johnson, \emph{Computers and Intractability: A Guide to the Theory
of NP-Completeness}, New York: W.H. Freeman. ISBN 0-7167-1045-5, 1979.

\bibitem{Ha08} E. Hainry, \emph{Reachability in Linear Dynamical Systems},
Computability in Europe 2008, Lecture Notes in Computer Science, LNCS 5028, Springer, 2008.

\bibitem{HHH} V. Halava, T. Harju, M. Hirvensalo, \emph{Positivity of
Second Order Linear Recurrent Sequences}, Discrete Applied Math.
154, Elsevier, pp. 447-451, 2006.

\bibitem{HHHK} V. Halava, T. Harju, M. Hirvensalo, J. Karhum\"aki,
\emph{Skolem's Problem - On the Border between Decidability and Undecidability},
TUCS Technical Report Number $683$, 2005.

\bibitem{HH07} V. Halava, M. Hirvensalo, \emph{Improved Matrix Pair
Undecidability Results}, Acta Informatica 44, 3:191-205, 2007.

\bibitem{Ka43} M. Kac, \emph{On the Distribution of Values of
Trigonometric Polynomials with Linearly Independent Frequencies},
Am. J. Math., 65, 609-615, 1943.

\bibitem{KL86} R. Kannan, R. Lipton, \emph{Polynomial-Time Algorithm
for the Orbit Problem}, Journal of the Association for Computing Machinery,
Volume 33, Issue 4, 808 - 821, 1986.

\bibitem{Le} C. Lech \emph{A Note on Recurring Sequences} Ark. Mat. 2,
417-421, 1953.

\bibitem{Ma} K. Mahler, \emph{Eine Arithmetische Eigenschaft der
Taylor-Koeffizienten Rationaler Funktionen}, Proc. Akad. Wet.,
Amsterdam 38, 50-60, 1935.

\bibitem{Mo73} C. J. Moreno, \emph{The Zeros of Exponential Polynomials},
Comp. Math., vol. 26, 69-78, 1973.

\bibitem{NS82} J. D. Nulton, K. B. Stolarsky, \emph{The First Sign Change
of a Cosine Polynomial}, Proc. of the Amer. Math. Soc., 84, No. 1, 55-59, 1982.

\bibitem{Po} G. P\'{o}lya, \emph{Geometrisches \"{u}ber die Verteilung
der Nullstellen gewisser ganzer Transzendenter Funktionen},
{M}\"{u}nch. {S}itzungsber., 50, 285-290, 1920.


\bibitem{Sch25} E. Schwengler, \emph{Geometrisches \"uber die Verteilung
der Nullstellen Spezieller Ganzer Funktionen}, Dissertation, ETH Z\"urich, 1925.

\bibitem{Sk} T. Skolem, \emph{Ein Verfahren zur Behandlung gewisser
exponentialer Gleichungen und diophantischer Gleichungen}, 8. Skand.
Mat. Kongr., Stockholm, 163-188, 1934.

\bibitem{Ta25} J. Tamarkin, \emph{Some General Problems of the Theory of
Ordinary Linear Differential Equations and Expansion of an Arbitrary
Function in Series of Fundamental Functions}, Mathematische Zeitschrift, 27,
Num. 1, 1-54, 1925.

\bibitem{vdP75} A. J. van der Poorten, \emph{A Note on the Zeros of
Exponential Polynomials}, Compositio Mathematica, 31 no. 2, 109-113, 1975.

\bibitem{V85} N. Vereshchagin, \emph{Occurrence of Zero in a
Linear Recursive Sequence}, Math. Notes, 38, 2:609-615, 1985.

\bibitem{Vo76} M. Voorhoeve, \emph{On the Oscillation of Exponential
Polynomials} Mat. Z., 151, No. 3, 217-293, 1976.

\bibitem{Wi17} C. E. Wilder, \emph{Expansion Problems of Ordinary Linear
Differential Equations with Auxiliary Conditions at More than Two Points},
Trans. of Math. Soc., 18, 415-442, 1917.

\bibitem{Wi72} J.C. Willems, \emph{Dissipative Dynamical Systems - Part II: Linear
Systems with Quadratic Supply Rates}, Archive for Rational Mechanics and Analysis
Vol. 45, 352-393, 1972.

\end{thebibliography}
\end{document}